\documentclass[letterpaper, 10 pt, conference]{ieeeconf}

\IEEEoverridecommandlockouts                              % This command is only needed if
\usepackage{graphicx} % for pdf, bitmapped graphics files
\usepackage{amsmath,amssymb}
\usepackage{mathtools}
\usepackage{array}
\usepackage{booktabs}
\usepackage[compress,noadjust]{cite}
\usepackage{textcomp}
\usepackage{hyperref}
\usepackage{tikz,pgfplots}
\usetikzlibrary{positioning}
\pgfplotsset{compat=1.18}
\usepgfplotslibrary{groupplots}

\newtheorem{theorem}{Theorem}
\newtheorem{lemma}{Lemma}
\newtheorem{proposition}{Proposition}
\newtheorem{corollary}{Corollary}

\newcommand{\Ne}{\mathbb{N}}
\newcommand{\Ze}{\mathbb{Z}}
\renewcommand{\Re}{\mathbb{R}}
\newcommand{\Qe}{\mathbb{Q}}
\newcommand{\jsr}{\mathrm{jsr}}
\newcommand{\ctilde}{\tilde{c}}
\newcommand{\Jtilde}{\tilde{J}}
\newcommand{\calA}{\mathcal{A}}
\newcommand{\calV}{\mathcal{V}}

\newif\ifextended
\extendedtrue % <-- uncomment for extended version

\title{\LARGE\bfseries On the differentiability of the value function of switched linear systems under arbitrary and controlled switching}

\author{Guillaume O.~Berger$^{1}$% <-this % stops a space
\thanks{$^{1}$G.~Berger is with ICTEAM, UCLouvain, Belgium.
He is an FNRS postdoctoral researcher.
{\tt\small guillaume.berger@uclouvain.be}}}

\begin{document}

\maketitle
\thispagestyle{empty}
\pagestyle{empty}

%%%%%%%%%%%%%%%%%%%%%%%%%%%%%%%%%%%%%%%%%%%%%%%%%%%%%%%%%%%%%%%%%%%%%%%%%%%%%%%%
\begin{abstract}
This paper studies the differentiability of the value function of switched linear systems under arbitrary switching and controlled switching, referred to as worst-case and optimal value functions respectively.
First, we show that the value functions are Lipschitz continuous, when the cost function is Lipschitz continuous.
Then, as the central contribution of this work, we show with examples that each of these functions can be non-differentiable on dense subsets of the state space, even if the cost function is smooth and Lipschitz continuous.
This has implications for optimal control and reinforcement learning since it implies that the exact computation of these value functions requires templates involving functions that are non-differentiable on dense subsets.
\end{abstract}

%%%%%%%%%%%%%%%%%%%%%%%%%%%%%%%%%%%%%%%%%%%%%%%%%%%%%%%%%%%%%%%%%%%%%%%%%%%%%%%%
\section{Introduction}

Switched linear systems are multi-modal dynamical systems wherein each mode is a linear system.
These systems appear naturally in a wide range of applications---e.g., mechanical systems with impact, digital power converters, cyber-physical systems, etc.---and as abstractions of more complex dynamical or hybrid systems~\cite{liberzon2003switching,jungers2009thejoint,sun2011stability}.
The stability and stabilizability theory of switched linear systems has been extensively studied: despite intrinsic challenges (undecidability, NP-hardness, non-algebraicity, etc.~of computing stabilizing controllers or stability certificates)~\cite{blondel1999complexity,jungers2009thejoint}, several powerful tools (e.g., based on piecewise, multiple or sum-of-squares Lyapunov functions) have been provided for the stability analysis and stabilization of switched linear systems; see, e.g.,~\cite{sun2011stability,ahmadi2014joint,lin2009stability} for recent surveys.

In this paper, we first focus on the problem of optimal control of switched linear systems under controlled switching, as well as the worst-case cost analysis of switched linear systems under arbitrary switching.
This means that we consider a cost function mapping states to cost values, and we aim to optimize the \emph{cost-to-go} of any trajectory of the system by choosing the switching law appropriately, or estimating the worst-case \emph{cost-to-go} of any trajectory of the system under all possible switching sequences.
Here, the \emph{cost-to-go} refers the sum of all state costs of the trajectory.
This problem has received a lot of attention in the literature~\cite{rantzer2005onapproximate,zhang2009onthevalue,gorges2011optimal,zhang2012infinitehorizon,zhu2015optimal,antunes2017linear,wu2020linear,wu2020optimal,hou2024animproved}, due its importance in applications and its challenging nature.
In particular,~\cite{rantzer2005onapproximate,gorges2011optimal,zhang2012infinitehorizon,antunes2017linear,hou2024animproved} propose efficient techniques to approximate the \emph{optimal value function} (the function mapping states of the state space to the smallest cost-to-go among all trajectories starting from them), and~\cite{wu2020optimal} proposes a polynomial-time algorithm for computing this function under additional assumptions on the system.
Nevertheless, the question of the complexity of the optimal value function has not been properly addressed in the literature yet (see ``related work'' below).
This question is however crucial if one aims to leave the realm of (efficient or not) approximate methods.

This paper provides a pessimistic answer to the above question.
Indeed, after showing that the optimal and worst-case value functions are Lipschitz continuous if the cost function is Lipschitz continuous, we show that it may however be non-differentiable on a dense subset of the state space, even if the cost function is smooth (e.g.,~quadratic).
This finding is deceptive for the objective of exact computation of the value functions of switched linear systems because it implies that iterative algorithms have almost zero chance of finding them in finite time (since they usually deal with piecewise smooth functions).
We further show that this negative result holds also for systems whose matrices have rational entries, so that it holds with nonzero probability when working with matrices encoded with finite precision.
We obtained these results by building a switched linear system for which the optimal (or worst-case) switching law leads to a closed-loop system that is shown to be an \emph{Interval Exchange Map}~\cite{keane1975interval}.
This well-studied class of dynamical systems is known for their topological transitivity (existence of a dense orbit); a property that we leverage to obtain the non-differentiability on dense subsets.

\subsection*{Related work}

The paper~\cite{zhang2009onthevalue} shows that the \emph{finite-horizon} optimal value function of the linear quadratic discrete-time switched LQR problem is piecewise quadratic.
It also shows that the finite-horizon optimal value function converges exponentially to the (infinite-horizon) optimal value function when the horizon tends to infinity.
However, it does not discuss whether the sequence of piecewise quadratic functions converges toward a piecewise differentiable function or not.

% In fact, the exponential growth of the number of pieces in the finite-horizon optimal cost-to-go is the main bottleneck in computing the optimal cost-to-go~\cite{zhang2009efficient,wu2020linear}; pruning methods are proposed in~\cite{hou2024animproved}, and relaxed dynamic programming techniques in~\cite{rantzer2005onapproximate,gorges2011optimal,zhang2012infinitehorizon}.

The paper~\cite{harder2024onthecontinuity} discusses the continuity and smoothness of the value function for autonomous dynamical systems.
In particular, conditions are given under which the value function is Lipschitz continuous, and it is shown that there are systems for which the value function is nowhere differentiable.
These results differ from ours because 1) they focus on autonomous systems, 2) they assume a discount factor in the cost-to-go, and 3) they rely on a nonlinear system to prove the non-differentiability.
% These results can however not be applied straightforwardly to our setting because 1) they apply to autonomous systems (not controlled or adversarial), and 2) the example of a nowhere differentiable value function relies on a chaotic nonlinear dynamical system.
Nevertheless, our proof of the Lipschitz continuity of the value functions of switched linear systems uses similar arguments to that in~\cite{harder2024onthecontinuity}.
This result, which is not the main contribution of this work, is used in the proof of our main result on the possible non-differentiability of the value functions of switched linear systems on dense subsets.
The proof of the latter is radically different from that of the akin result in~\cite{harder2024onthecontinuity}.
In particular, we stress that the value functions of switched linear systems are differentiable almost everywhere (consequence of being Lipschitz continuous, by Rademacher's theorem~\cite{rudin1987real}), which is a striking difference with the example in~\cite{harder2024onthecontinuity}.
Our proof relies on the existence of a switched linear system for which the optimal or worst-case switching law produces a closed-loop system called an \emph{Interval Exchange Map}, which is known to be topologically transitive~\cite{keane1975interval}; a fact that we leverage to obtain the non-smoothness.

% Such a function cannot exist for switched linear systems because we show that the optimal and worst-case value functions are Lipschitz continuous, so that they are differentiable almost everywhere.\footnote{We remind that a function can be differentiable almost everywhere (for the Lebesgue measure), and still be non-differentiable on a dense subset of its domain, since there exist dense subsets that have Lebesgue measure zero (e.g., $\mathbb{Q}$ as a subset of $\Re$).}
% Nevertheless, to the best of our knowledge, the question of the Lipschitz continuity and smoothness of the optimal value function has not been addressed yet.

\emph{Notation.}
$\Ne$ is the set of nonnegative integers.
Given $m\in\Ne_{>0}$, $[m]$ is the set $\{1,\ldots,m\}$.
We denote the Euclidean vector norm by $\lVert\cdot\rVert$ (i.e., $\lVert x\rVert^2=x^\top x$).

\section{Problem statement}

\subsection{Switched linear systems}

We consider a discrete-time switched linear system of the form:
\begin{equation}\label{eq:sls-def}
\xi(t+1) = A_{\sigma(t)}\xi(t), \quad t\in\Ne,
\end{equation}
where $\xi(t)\in\Re^n$ is the \emph{state} at time $t$, $\sigma(t)\in[m]$ is the \emph{mode} at time $t$, and for all $i\in[m]$, $A_i\in\Re^{n\times n}$ represents the \emph{transition matrix} of mode $i$.
The function $\sigma:\Ne\to[m]$ mapping times to modes is called the \emph{switching signal}, which can be arbitrary or controlled.
The system~\eqref{eq:sls-def} is denoted by $\{A_i\}_{i=1}^m$.
Given $x\in\Re^n$, $\sigma:\Ne\to[m]$ and $t\in\Ne$, we denote by $\xi(t,x,\sigma)$ the solution at time $t$ of~\eqref{eq:sls-def} with signal $\sigma$ and initial state $\xi(0)=x$.

In this paper, we restrict our attention to \emph{stable} switched linear systems.
This means that the \emph{joint spectral radius}~\cite{jungers2009thejoint}, which represents the worst-case linear rate of growth of the trajectories of the system, is smaller than one.
Formally, the joint spectral radius of $\calA\coloneqq\{A_i\}_{i=1}^m$ is defined by
\begin{align*}
\jsr(\calA) &= \inf\,\{r\geq0:\exists\,C\geq0\:\text{s.t.}\\
&\quad\forall\,\xi\;\text{solution of~\eqref{eq:sls-def}},\\
&\quad\forall\,t\in\Ne,\:\lVert\xi(t)\rVert\leq Cr^t\lVert\xi(0)\rVert\},
\end{align*}
and we make the standing assumption that $\jsr(\{A_i\}_{i=1}^m)<1$.
The following result will be instrumental in our proofs:

\begin{theorem}[{\cite{jungers2009thejoint}}]\label{thm:norm-jsr}
Let $\calA\coloneqq\{A_i\}_{i=1}^m$ be a switched linear system and $\rho>\jsr(\calA)$.
There is a norm $\lVert\cdot\rVert_*$ such that for all solutions $\xi$ of~\eqref{eq:sls-def} and all $t\in\Ne$, it holds that $\lVert\xi(t)\rVert_*\leq\rho^t\lVert\xi(0)\rVert_*$.
\end{theorem}

Given a norm $\lVert\cdot\rVert'$, we denote its \emph{eccentricity} by
\[
\kappa(\lVert\cdot\rVert')=\frac{\max_{\lVert x\rVert=1}\lVert x\rVert'}{\min_{\lVert x\rVert=1}\lVert x\rVert'}.
\]

\subsection{Optimal and worst-case value functions}

We let $c:\Re^n\to\Re_{\geq0}$ be a cost function mapping states to cost values.
We assume that $c(0)=0$ and that $c$ is locally Lipschitz continuous around $0$.\footnote{Together with the stability assumption, this ensures that $J$, $J^\star$ and $J^\circ$ (defined just after) are finite everywhere.}
Given a signal $\sigma:\Ne\to[m]$ and $x\in\Re^n$, the \emph{cost-to-go} of the trajectory starting from $x$ with signal $\sigma$ is defined by
\[
J(x,\sigma)=\sum_{t=0}^\infty c(\xi(t,x,\sigma)).
\]
In the context of controlled switching, the goal is to find the switching signal (which may depend on $x$) such that the associated cost-to-go is minimal.
Hence, we define the \emph{optimal value function} as
\[
J^\star(x)=\inf_{\sigma:\Ne\to[m]}\sum_{t=0}^\infty c(\xi(t,x,\sigma)).
\]
This can be equivalently seen as the smallest cost-to-go among all trajectories starting from $x$.
On the other hand, in the context of arbitrary switching, we are interested in the largest cost-to-go of a trajectory starting from a given $x$.
Hence, we define the \emph{worst-case value function} as
\[
J^\circ(x)=\sup_{\sigma:\Ne\to[m]}\sum_{t=0}^\infty c(\xi(t,x,\sigma)).
\]

In this paper, we first show that $J^\star$ and $J^\circ$ are Lipschitz continuous when $c$ is Lipschitz continuous
Then, we show that they can be non-differentiable on dense subsets of the state space, even if $c$ is differentiable everywhere.

The following well-known \emph{dynamic programming equations} (also called \emph{Bellman equations}; see, e.g.,~\cite[Theorem~3.1]{meyn2022control}) will be useful in the proofs:

\begin{proposition}
For a switched linear system $\{A_i\}_{i=1}^m$ and a cost function $c$, let $J^\star$ and $J^\circ$ be the associated optimal and worst-case value functions respectively.
It holds that
\begin{align}
J^\star(x) &= c(x) + \min_{i\in[m]} J^\star(A_ix), \label{eq:dp-optimal} \\
J^\circ(x) &= c(x) + \max_{i\in[m]} J^\circ(A_ix). \label{eq:dp-worst}
\end{align}\vskip0pt
\end{proposition}

\section{Lipschitz continuity of the value functions}

In this section, we show that the optimal and worst-case value functions are Lipschitz continuous if the cost function is Lipschitz continuous.
As a reminder, a function $f:\Re^n\to\Re$ is \emph{Lipschitz continuous} on a set $X\subseteq\Re^n$ with Lipschitz constant $L$ if for all $x,y\in X$, it holds that
\[
\lvert f(x)-f(y)\rvert\leq L\lVert x-y\rVert.
\]

\begin{theorem}\label{thm:lipschitz}
Let $\calA\coloneqq\{A_i\}_{i=1}^m\subseteq\Re^{n\times n}$ be a switched linear system and $c:\Re^n\to\Re_{\geq0}$ a cost function.
Let $J^\star$ and $J^\circ$ be the associated optimal and worst-case value functions respectively.
Assume that $\jsr(\calA)<1$ and $c$ is Lipschitz continuous on a neighborhood $X\subseteq\Re^n$ of the origin, with Lipschitz constant $L$.
Let $\rho<1$ and $\lVert\cdot\rVert_*$ be as in Theorem~\ref{thm:norm-jsr}.
Let $\eta>0$ be such that $B_*(\eta)\coloneqq\{x\in\Re^n:\lVert x\rVert_*\leq\eta\}\subseteq X$.
It holds that $J^\star$ and $J^\circ$ are Lipschitz continuous on $B_*(\eta)$, with Lipschitz constant $M=\kappa(\lVert\cdot\rVert_*)L/(1-\rho)$.
\end{theorem}

% The proof relies on standard arguments (e.g., similar to those in~\cite{harder2024onthecontinuity}), and leverages the exponential stability of the system:

\begin{proof}
We do the proof for $J^\star$, and it extends straightforwardly to $J^\circ$.
Denote $r=\kappa(\lVert\cdot\rVert_*)$ and $M=rL/(1-\rho)$.
Let $x,y\in B_*(\eta)$.
Let $\epsilon>0$.
We will show that $J^\star(x)\leq J^\star(y)+M\lVert x-y\rVert+\epsilon$.
Since $x$, $y$ and $\epsilon$ are arbitrary, this will conclude the proof.
Let $\sigma$ be a switching signal such that $J(y,\sigma)\leq J^\star(y)+\epsilon$.
For all $t\in\Ne$, it holds that
\[
\lVert\xi(t,x,\sigma)-\xi(t,y,\sigma)\rVert_*=\lVert\xi(t,x-y,\sigma)\rVert_*\leq\rho^t\lVert x-y\rVert_*.
\]
Hence, it holds that for all $t\in\Ne$,
\[
\lVert\xi(t,x,\sigma)-\xi(t,y,\sigma)\rVert\leq r\rho^t\lVert x-y\rVert.
\]
Furthermore, for all $t\in\Ne$ and $z\in\{x,y\}$,
\[
\lVert\xi(t,z,\sigma)\rVert_*\leq\rho^t\lVert z\rVert_*\leq\eta.
\]
Thus, by the assumption on $c$, we obtain that for all $t\in\Ne$,
\[
c(\xi(t,x,\sigma))\leq c(\xi(t,y,\sigma))+r\rho^tL\lVert x-y\rVert.
\]
It follows that
\begin{align*}
J(x,\sigma) &= \sum_{t=0}^\infty c(\xi(t,x,\sigma)) \\
&\leq \sum_{t=0}^\infty c(\xi(t,y,\sigma)) + \frac{rL}{1-\rho} \lVert x-y\rVert \\
&= J(y,\sigma) + \frac{rL}{1-\rho} \lVert x-y\rVert.
\end{align*}
Hence, we get $J^\star(x)\leq J^\star(y)+\epsilon+M\lVert x-y\rVert$, concluding the proof.
\end{proof}

As a corollary, we obtain that $J^\star$ and $J^\circ$ are differentiable almost everywhere.
This is in stark contrast with the example in~\cite[Proposition~1]{harder2024onthecontinuity} where the value function of a nonlinear dynamical system is nowhere differentiable.

\begin{corollary}\label{cor:almost-differentiable}
Let $\calA\coloneqq\{A_i\}_{i=1}^m\subseteq\Re^{n\times n}$ be a switched linear system and $c:\Re^n\to\Re_{\geq0}$ a cost function.
Let $J^\star$ and $J^\circ$ be the associated optimal and worst-case value functions respectively.
Assume that $\jsr(\calA)<1$ and $c$ is Lipschitz continuous on a neighborhood of the origin.
It holds that $J^\star$ and $J^\circ$ are differentiable almost everywhere (for the Lebesgue measure) on a neighborhood of the origin.
\end{corollary}

\begin{proof}
Rademacher's theorem, a standard result in real analysis (see, e.g.,~\cite[Theorem~7.20]{rudin1987real}), states that locally Lipschitz continuous functions are differentiable almost everywhere (for the Lebesgue measure).
\end{proof}

\section{Possible non-differentiability of the value functions on dense subsets}

In this section, we prove the main result of this paper, stating that there exist switched linear systems for which the optimal or worst-case value function is non-differentiable on a dense subset of the state space, even if the cost function is smooth everywhere.\footnote{%
We note that this result is \emph{not} in contradiction with Corollary~\ref{cor:almost-differentiable} since there are dense subsets of $\Re^n$ that have zero Lebesgue measure (e.g., $\Qe^n\subseteq\Re^n$).
Said otherwise, the subtle point is that the value function can be differentiable almost everywhere (for the Lebesgue measure) and still be non-differentiable on a dense set.}

\begin{theorem}\label{thm:non-differentiable}
Let $n\in\Ne_{\geq2}$.
There exist a switched linear system $\calA\coloneqq\{A_i\}_{i=1}^m\subseteq\Re^{n\times n}$ and a cost function $c:\Re^n\to\Re_{\geq0}$ such that $\jsr(\calA)<1$, $c$ is $C^\infty$ on $\Re^n$ and Lipschitz continuous on a neighborhood of the origin, and $J^\star$ and $J^\circ$ are non-differentiable on dense subsets of $\Re^n$, where $J^\star$ and $J^\circ$ are the associated optimal and worst-case value functions respectively.
Furthermore, the matrices in $\calA$ can be assumed to have rational entries, i.e., $\calA\subseteq\Qe^{n\times n}$.
\end{theorem}

\subsection{Proof of Theorem~\ref{thm:non-differentiable}}

The proof consists in building a switched linear system and a cost function that satisfy the requirements of the theorem.
First, we consider the case $n=2$, and then we generalize.

\subsubsection{Case $n=2$}

The construction for $n=2$ is as follows.
Consider two angles,
\[
\alpha\in(\pi/8,3\pi/8) \;\; \text{and} \;\; \beta=\alpha-\pi/2\in(-3\pi/8,-\pi/8),
\]
such that $\alpha/\pi\notin\Qe$.
The system consists of two matrices, which are \emph{scaled rotations}, of angle $\alpha$ and $\beta$ respectively, and scaling factor $\rho$:
\begin{align}
A_1 &= \rho\left[\begin{array}{cc}
\cos(\alpha) & -\sin(\alpha) \\
\sin(\alpha) & \cos(\alpha)
\end{array}\right], \label{eq:A1} \\
A_2 &= \rho\left[\begin{array}{cc}
\cos(\beta) & -\sin(\beta) \\
\sin(\beta) & \cos(\beta)
\end{array}\right], \label{eq:A2}
\end{align}
where $\rho=0.01$.
We denote $\calA=\{A_1,A_2\}$.
We note that $\alpha$ and $\beta$ can be chosen so that $A_1$ and $A_2$ have rational entries; see \ifextended Lemma~\ref{lem:rational} at the end of this section\else the extended version~\cite{berger2025onthedifferentiability}\fi.
The following properties of the system will be useful.%
\footnote{All free variables appearing in the results (lemmas, corollaries, etc.)~in the rest of this section refer to quantities defined in the body of the text.}

\begin{lemma}\label{lem:sys-stable}
The system $\calA$ is stable with $\jsr(\calA)=\rho$, and the norm $\lVert\cdot\rVert_*=\lVert\cdot\rVert$ satisfies the requirements of Theorem~\ref{thm:norm-jsr} with this $\rho$.
\end{lemma}

\begin{proof}
It holds that for all $x\in\Re^2$ and $A\in\calA$, $\lVert Ax\rVert=\rho\lVert (A/\rho)x\rVert=\rho\lVert x\rVert$ where we used that $A/\rho$ is a rotation matrix.
Hence, for all solution $\xi$ of system $\calA$, it holds that $\lVert\xi(t)\rVert=\rho^t\lVert\xi(0)\rVert$.
\end{proof}

As for the cost function, we let
\begin{equation}\label{eq:cost}
c(x)=x_1^2+2x_2^2, \quad \text{where} \quad x=[x_1,x_2]^\top.
\end{equation}
The function $c$ is $C^\infty$ on $\Re^2$ and Lipschitz continuous on the set $B\coloneqq\{x\in\Re^2:\lVert x\rVert\leq1\}$, with Lipschitz constant $L=4$.
Moreover, $c$ is homogeneous of degree $2$, meaning that for all $\lambda\in\Re$ and $x\in\Re^2$, $c(\lambda x)=\lambda^2c(x)$.

Let $J^\star$ be the optimal value function of $\calA$ with cost $c$.
We do the analysis of $J^\star$ below, in particular showing in several steps that it is non-differentiable on a dense subset of $\Re^2$.
(The worst-case value function will be discussed later, with similar arguments.)

First, by Theorem~\ref{thm:lipschitz}, $J^\star$ is Lipschitz continuous:

\begin{lemma}\label{lem:sys-lipschitz}
The function $J^\star$ is Lipschitz continuous on $B$, with Lipschitz constant $M=L/(1-\rho)$.
\end{lemma}

\begin{proof}
Direct from Theorem~\ref{thm:lipschitz} and Lemma~\ref{lem:sys-stable}.
\end{proof}

Second, a consequence of $c$ being homogeneous is that $J^\star$ is homogeneous:

\begin{lemma}\label{lem:sys-homogeneous}
The function $J^\star$ is homogeneous of degree $2$, i.e., for all $\lambda\in\Re$ and $x\in\Re^2$, $J^\star(\lambda x)=\lambda^2J^\star(x)$.
\end{lemma}

\begin{proof}
It suffices to prove that for all $x\in\Re^2$, $\lambda\neq0$ and $\epsilon>0$, $J^\star(\lambda x)\leq \lambda^2J^\star(x)+\epsilon$.
Therefore, let $x\in\Re^2$, $\lambda\neq0$ and $\epsilon>0$.
Let $\sigma$ be a switching signal such that $J(x,\sigma)\leq J^\star(x)+\lambda^{-2}\epsilon$.
By the linearity, it holds that for all $t\in\Ne$, $\xi(t,\lambda x,\sigma)=\lambda\xi(t,x,\sigma)$.
Hence, $J(\lambda x,\sigma)=\lambda^2J(x,\sigma)$, so that $J^\star(\lambda x)\leq \lambda^2J^\star(x)+\epsilon$, concluding the proof.
\end{proof}

Given $\theta\in\Re$, we let $x(\theta)=[\cos(\theta),\sin(\theta)]^\top$.
We also let
\[
\ctilde(\theta)\coloneqq c(x(\theta))=\cos^2(\theta)+2\sin^2(\theta)=1+\sin^2(\theta).
\]
See Fig.~\ref{fig:ctilde} for an illustration.

\begin{lemma}\label{lem:prop-ctilde}
For all $x\in[\pi/8,3\pi/8]+\pi\Ze$, $\ctilde'(x)\geq\sqrt{2}/2$, and for all $x\in[-3\pi/8,-\pi/8]+\pi\Ze$, $\ctilde'(x)\leq-\sqrt{2}/2$.
\end{lemma}

\begin{proof}
Straightforward from $\ctilde'(\theta)=\sin(2\theta)$.
\end{proof}

For each $\theta\in\Re$, we let $\Delta\ctilde(\theta)\coloneqq\ctilde(\theta+\alpha)-\ctilde(\theta+\beta)$.
We also let $\mu\coloneqq-(\alpha+\beta)/2=\pi/4-\alpha=-\pi/4-\beta$.

\begin{lemma}\label{lem:prop-delta-ctilde}
It holds that $\Delta\ctilde(\theta)=\sin(2(\theta-\mu))$.
\end{lemma}

\begin{proof}
We have that
\begin{align*}
\ctilde(\theta+\alpha) &= 1+\sin^2(\theta+\pi/4-\mu) \\
&= 1+(1-\cos(2(\theta-\mu)+\pi/2))/2 \\
&= 1+(1+\sin(2(\theta-\mu)))/2.
\end{align*}
Similarly,
\begin{align*}
\ctilde(\theta+\beta) &= 1+\sin^2(\theta-\pi/4-\mu) \\
&= 1+(1-\cos(2(\theta-\mu)-\pi/2))/2 \\
&= 1+(1-\sin(2(\theta-\mu)))/2.
\end{align*}
Hence, $\Delta\ctilde(\theta)=\sin(2(\theta-\mu))$.
\end{proof}

\begin{figure}
    \centering
    \begin{tikzpicture}
        \begin{axis}[
                width=0.85\linewidth,
                height=5cm,
                axis y line*=right,
                axis x line*=bottom,
                xmin=-pi,xmax=pi,
                xtick={-3*pi/4,-pi/2,-pi/4,0,pi/4,pi/2,3*pi/4},
                xticklabels={$\frac{-3\pi}{4}$,$\frac{-\pi}{2}$,$\frac{-\pi}{4}$,$0$,$\frac{\pi}{4}$,$\frac{\pi}{2}$,$\frac{3\pi}{4}$},
                xlabel={$\theta$},
                xlabel style={inner sep=0pt,outer sep=0pt},
                ymin=-1.2,
                ymax=1.2,
                ylabel={$\Delta\tilde{J}$},
                ylabel style={inner sep=0pt,outer sep=0pt},
                ylabel style={inner sep=0pt,outer sep=0pt},
                grid=major,
                legend pos=south east,
                clip=true
            ]
            \pgfmathsetmacro{\r}{0.6}
            \pgfmathsetmacro{\s}{0.6-pi/2}
            \addplot[domain=-pi:pi,samples=100,green!70!black]{sin(deg(x+\r))^2+0.03*max(sin(30*deg(x+\r)),cos(30*deg(x+\r)))-sin(deg(x+\s))^2-0.03*max(sin(30*deg(x+\s)),cos(30*deg(x+\s)))};
            \addlegendentry{$\Delta\Jtilde^\star$};
            \draw[purple,dashed] (axis cs:pi/4-0.6, -1.2) -- (axis cs:pi/4-0.6, 1.2);
            \node[fill=black,circle,minimum width=2.5pt,inner sep=0pt] at (axis cs:0.2, 0) {};
            \node at (axis cs:0.05, 0.1) {$\nu$};
            \node[fill=black,circle,minimum width=2.5pt,inner sep=0pt] at (axis cs:1.77, 0) {};
            \node at (axis cs:1.9, 0.1) {$\omega$};
            \node[fill=black,circle,minimum width=2.5pt,inner sep=0pt] at (axis cs:1.77-pi, 0) {};
            \node at (axis cs:1.3-pi, -0.13) {$\omega-\pi$};
        \end{axis}
        \begin{axis}[
                width=0.85\linewidth,
                height=5cm,
                axis y line*=left,
                axis x line=none, % share x-axis
                xmin=-pi,xmax=pi,
                ymin=0.9,
                ymax=2.1,
                ylabel={$\ctilde$, $\Jtilde^\star$},
                ylabel style={inner sep=0pt,outer sep=0pt},
                legend pos=south west,
                clip=true,
            ]
            \addplot[domain=-pi:pi,samples=100,red]{1+sin(deg(x))^2};
            \addlegendentry{$\ctilde$};
            \addplot[domain=-pi:pi,samples=100,blue]{1+sin(deg(x))^2+0.03*max(sin(30*deg(x)),cos(30*deg(x)))};
            \addlegendentry{$\Jtilde^\star$};
            \addplot[domain=-pi:pi,samples=100,red]{1+sin(deg(x))^2};
        \end{axis}
    \end{tikzpicture}
    \vskip-5pt
    \caption{The function $\ctilde$ and a potential plot of $\Jtilde^\star$ (left y-axis), along with $\Delta\Jtilde^\star$ (right y-axis) for $\alpha=0.6$ and $\beta=\alpha-\pi/2$ ($\mu$ is the purple line).}
    \label{fig:ctilde}
\end{figure}

Finally, for each $\theta\in\Re$, we let $\Jtilde^\star(\theta)\coloneqq J^\star(x(\theta))$.

\begin{lemma}\label{lem:sys-dp}
For all $\theta\in\Re$, it holds that
\[
\Jtilde^\star(\theta)=\ctilde(\theta) + \rho^2\min\{\Jtilde^\star(\theta+\alpha),\Jtilde^\star(\theta+\beta)\}.
\]\vskip0pt
\end{lemma}

\begin{proof}
For any $\theta\in\Re$, it holds that $A_1x(\theta)=\rho x(\theta+\alpha)$ and $A_2x(\theta)=\rho x(\theta+\beta)$ (because $A_1$ and $A_2$ are scaled rotations).
Hence, we obtain the conclusion by using~\eqref{eq:dp-optimal} and Lemma~\ref{lem:sys-homogeneous}.
\end{proof}

For each $\theta\in\Re$, we let $\Delta\Jtilde^\star(\theta)\coloneqq\Jtilde^\star(\theta+\alpha)-\Jtilde^\star(\theta+\beta)$.
Denote $\delta=0.01$.
The following result is \emph{key} (see also Fig.~\ref{fig:ctilde} for an illustration):

\begin{lemma}\label{lem:delta-jtilde-well}
There exist
\begin{itemize}
    \item $\nu\in(\mu-\delta,\mu+\delta)$ such that $\Delta\Jtilde^\star(\nu)=0$;
    \item $\omega\in(\mu+\pi/2-\delta,\mu+\pi/2+\delta)$ such that $\Delta\Jtilde^\star(\omega)=0$.
\end{itemize}
Moreover, $\Delta\Jtilde^\star$ is
\begin{itemize}
    \item negative on $(\omega-\pi,\nu)+\pi\Ze$;
    \item positive on $(\nu,\omega)+\pi\Ze$.
\end{itemize}\vskip0pt
\end{lemma}

\begin{proof}
First of all, let us remind that by Lemma~\ref{lem:sys-lipschitz}, $\Jtilde^\star$ is Lipschitz continuous on $\Re$, with Lipschitz constant $M$, and for all $\theta\in\Re$, $\lvert\Jtilde^\star(\theta)\rvert\leq M$, where $M<4.1$.%
\footnote{Indeed, recall that $L=4$ and $\rho=0.01$.}

\emph{Negativity on the set $S_-\coloneqq[\mu-\pi/2+\delta,\mu-\delta]+\pi\Ze$:}
By Lemma~\ref{lem:prop-delta-ctilde}, we have that for all $\theta\in S_-$, $\Delta\ctilde(\theta)<-0.019$.
Hence, by Lemma~\ref{lem:sys-dp}, we obtain that for all $\theta\in S_-$,
\[
\Delta\Jtilde^\star(\theta)<-0.019+2\rho^2M<-0.018<0.
\]

\emph{Positivity on the set $S_+\coloneqq[\mu+\delta,\mu+\pi/2-\delta]+\pi\Ze$:}
By Lemma~\ref{lem:prop-delta-ctilde}, we have that for all $\theta\in S_+$, $\Delta\ctilde(\theta)>0.019$.
Hence, by Lemma~\ref{lem:sys-dp}, we obtain that for all $\theta\in S_+$,
\[
\Delta\Jtilde^\star(\theta)>0.019-2\rho^2M>0.018>0.
\]

\emph{Existence of $\nu$:}
From the above, we know that $\Delta\Jtilde^\star(\mu-\delta)<0$ and $\Delta\Jtilde^\star(\mu+\delta)>0$.
Hence, by continuity of $\Delta\Jtilde^\star$, there is $\nu\in(\mu-\delta,\mu+\delta)$ such that $\Delta\Jtilde^\star(\nu)=0$.

\emph{Existence of $\omega$:}
From the above, we know that $\Delta\Jtilde^\star(\mu+\pi/2-\delta)>0$ and $\Delta\Jtilde^\star(\mu+\pi/2+\delta)<0$.
Hence, there is $\omega\in(\mu+\pi/2-\delta,\mu+\pi/2+\delta)$ such that $\Delta\Jtilde^\star(\omega)=0$.

\emph{Strict increase on $S_\nearrow\coloneqq[\mu-\delta,\mu+\delta]+\pi\Ze$:}
We show that $\Delta\Jtilde^\star$ is increasing on $S_\nearrow$, which will imply that $\Delta\Jtilde^\star$ is negative on $[\mu-\delta,\nu)+\pi\Ze$ and positive on $(\nu,\mu+\delta]+\pi\Ze$.
Let $k\in\Ze$ and $\mu-\delta+k\pi\leq\theta_1<\theta_2\leq\mu+\delta+k\pi$.
By Lemma~\ref{lem:prop-delta-ctilde}, we have that $\Delta\ctilde(\theta_2)-\Delta\ctilde(\theta_1)\geq\theta_2-\theta_1$ (mean value theorem).
Hence, by Lemma~\ref{lem:sys-dp},
\[
\Delta\Jtilde^\star(\theta_2)-\Delta\Jtilde^\star(\theta_1)\geq(1-\rho^2M)(\theta_2-\theta_1)>0.
\]

\emph{Strict decrease on $S_\searrow\coloneqq[\mu+\pi/2-\delta,\mu+\pi/2+\delta]+\pi\Ze$:}
We show that $\Delta\Jtilde^\star$ is decreasing on $S_\searrow$, which will imply that $\Delta\Jtilde^\star$ is positive on $[\mu+\pi/2-\delta,\omega)+\pi\Ze$ and negative on $(\omega,\mu+\pi/2+\delta]+\pi\Ze$.
Let $k\in\Ze$ and $\mu+\pi/2-\delta+k\pi\leq\theta_1<\theta_2\leq\mu+\pi/2+\delta+k\pi$.
By Lemma~\ref{lem:prop-delta-ctilde}, we have that $\Delta\ctilde(\theta_2)-\Delta\ctilde(\theta_1)\leq\theta_1-\theta_2$ (mean value theorem).
Hence, by Lemma~\ref{lem:sys-dp},
\[
\Delta\Jtilde^\star(\theta_2)-\Delta\Jtilde^\star(\theta_1)\geq(\rho^2M-1)(\theta_2-\theta_1)<0.
\]

This concludes the proof.
\end{proof}

From now on, we let $\nu$ and $\omega$ be as Lemma~\ref{lem:delta-jtilde-well}.
We note that $[\nu+\beta,\nu+\alpha]\subseteq(\omega-\pi,\omega).$\footnote{Since $\nu+\alpha<\mu+3\pi/8+\delta<\mu+\pi/2-\delta$ and $\beta>\mu-3\pi/8-\delta>\mu-\pi/2+\delta$ (recall that $\delta=0.01$).}
Similarly, $[\mu-\delta,\mu+\delta]\subseteq(\omega-\pi,\omega)$.
We show that $\Jtilde^\star$ is not differentiable at $\nu$.

\begin{lemma}\label{lem:jtilde-non-differentiable}
The function $\Jtilde^\star$ is not differentiable at $\nu$.
\end{lemma}

\begin{proof}
By Lemma~\ref{lem:sys-dp} and $\ctilde$ being differentiable, it suffices to show that the function
\[
h:\theta\mapsto\min\{\Jtilde^\star(\theta+\alpha),\Jtilde^\star(\theta+\beta)\}
\]
is not differentiable at $\nu$.
For that, we will show that
\begin{itemize}
    \item[a)] for $\mu-\delta\leq\theta<\nu$, $h(\nu)-h(\theta)\geq\frac{1}2(\nu-\theta)$,
    \item[b)] for $\nu<\theta\leq\mu+\delta$, $h(\theta)-h(\nu)\leq\frac{-1}2(\theta-\nu)$.
\end{itemize}
This will imply that $h$ is not differentiable at $\nu$ since
\[
\limsup_{\theta\to\nu} \frac{h(\nu)-h(\theta)}{\nu-\theta}\geq\frac{1}2>\frac{-1}2\geq\liminf_{\theta\to\nu} \frac{h(\nu)-h(\theta)}{\nu-\theta}.
\]

\emph{Proof of a):}
Let $\theta\in[\mu-\delta,\nu)$.
By Lemma~\ref{lem:delta-jtilde-well}, it holds that $\Delta\Jtilde^\star(\theta)<0$, so that $\Jtilde^\star(\theta+\alpha)<\Jtilde^\star(\theta+\beta)$, and thus $h(\theta)=\Jtilde^\star(\theta+\alpha)$.
Similarly, $h(\nu)=\Jtilde^\star(\nu+\alpha)$.
We note that
\[
\pi/8\leq\pi/4-\delta\leq\theta+\alpha<\nu+\alpha\leq\pi/4+\delta\leq3\pi/8.
\]
(recall that $\mu=\pi/4-\alpha$).
Hence, by Lemma~\ref{lem:prop-ctilde}, we get that $\ctilde(\nu+\alpha)-\ctilde(\theta+\alpha)\geq\sqrt2/2(\nu-\theta)$ (mean value theorem).
Thus, by Lemmas~\ref{lem:sys-dp} and~\ref{lem:sys-lipschitz},
\[
\Jtilde^\star(\nu+\alpha)-\Jtilde^\star(\theta+\alpha)\geq(\sqrt2/2-\rho^2M)(\nu-\theta)\geq\frac12(\nu-\theta),
\]
which gives $h(\nu)-h(\theta)\geq\frac12(\nu-\theta)$.

\emph{Proof of b):}
Similar to the above: if $\theta\in(\nu,\mu+\delta]$, then we have that $h(\theta)=\Jtilde^\star(\theta+\beta)$, $h(\nu)=\Jtilde^\star(\nu+\beta)$ and
\[
-3\pi/8\leq-\pi/4-\delta\leq\nu+\beta<\theta+\beta\leq-\pi/4+\delta\leq-\pi/8.
\]
(recall that $\mu=-\pi/4-\beta$).
Hence, by Lemmas~\ref{lem:prop-ctilde},~\ref{lem:sys-dp} and~\ref{lem:sys-lipschitz}, we obtain $h(\theta)-h(\nu)\leq-\frac12(\theta-\nu)$.

This concludes the proof.
\end{proof}

Let us now consider the dynamical system $T$ on the set $I\coloneqq[\nu+\beta,\nu+\alpha)$ defined by%
\footnote{We sometimes note $T(\theta)=T\theta$ for convenience.}
\[
T:I\to I,\quad T(\theta) = \left\lbrace \begin{array}{ll}
\theta + \alpha & \text{if $\theta<\nu$,} \\
\theta + \beta & \text{if $\theta\geq\nu$.}
\end{array} \right.
\]
This system is an \emph{Interval Exchange Map (IEM)}~\cite{keane1975interval}, because it translates the interval $I_1\coloneqq[\nu+\beta,\nu)$ to the interval $I_1'\coloneqq[\nu+\beta+\alpha,\nu+\alpha)$, and the interval $I_2\coloneqq[\nu,\nu+\alpha)$ to $I_2'\coloneqq[\nu+\beta,\nu+\alpha+\beta)$.
These systems are known to be topologically transitive (existence of a dense orbit) when the length of $I_1$ (or $I_2$) is an irrational multiple of the length of $I$~\cite{keane1975interval}.
See Fig.~\ref{fig:iem} for an illustration.

\begin{figure}
    \centering
    \begin{tikzpicture}
        \pgfmathsetmacro{\xmin}{0.0}
        \pgfmathsetmacro{\xmax}{7.0}
        \pgfmathsetmacro{\xmid}{\xmax-0.6*(\xmax-\xmin)/(pi/2)}
        \pgfmathsetmacro{\ylab}{-0.5}
        \draw[-latex] (\xmin-0.1,0) -- (\xmax+0.3,0);
        \draw (\xmid,-0.1) -- (\xmid,0.1);
        \node at (\xmid,\ylab) {$\nu$};
        \node at (\xmin,0) {$[$};
        \node at (\xmin,\ylab) {$\nu+\beta$};
        \node at (\xmax,0) {$)$};
        \node at (\xmax,\ylab) {$\nu+\alpha$};
        \pgfmathsetmacro{\xrec}{\xmid}
        \pgfmathsetmacro{\dxp}{\xmax-\xmid}
        \pgfmathsetmacro{\dxm}{\xmin-\xmid}
        \pgfmathsetmacro{\NN}{15}
        \foreach \i in {1,...,\NN}
        {
            \pgfmathsetmacro{\cond}{(\xrec>=\xmid)}
            \pgfmathparse{\xrec+\cond*\dxm+(1-\cond)*\dxp}
            \xdef\xrec{\pgfmathresult} % GLOBAL update so next iteration sees it
            \node[fill=black,circle,minimum width=3pt,inner sep=0pt] (p-\i) at (\xrec,0) {};
        }
        \coordinate (p-0) at (\xmid,0);
        \pgfmathtruncatemacro{\MM}{\NN-1}
        \foreach \i in {0,...,\MM}
        {
            \pgfmathtruncatemacro{\j}{\i+1}
            \draw[blue] (p-\i) edge[-latex,bend left] (p-\j);
        }
    \end{tikzpicture}
    \vskip-5pt
    \caption{The truncated backward trajectory $\{T^{-1}\nu,\ldots,T^{-15}\nu\}$ (black dots) of the map $T$ for $\alpha=0.6$, $\beta=\alpha-\pi/2$ and $\nu=0.1$.
    The blue arrows represent transitions by $T^{-1}$.}
    \label{fig:iem}
\end{figure}

\begin{lemma}\label{lem:dense-orbit}
For any $\theta\in I$, the backward orbit of $T$ from $\theta$, i.e., the set $\{T^{-1}\theta,T^{-2}\theta,T^{-3}\theta,\ldots\}$, is dense in $I$.
\end{lemma}

\begin{proof}
Consider the map
\[
S:[0,\pi/2)\to[0,\pi/2),\quad S(\theta)=\theta+\alpha\!\mod\pi/2.
\]
The map $S$ is an irrational rotation (recall that $\alpha/\pi\notin\Qe$).
Hence, its forward and backward orbits are dense in $[0,\pi/2)$.
Now, observe that
\begin{equation}\label{eq:s-rel-t}
T(\theta)=S(\theta-\phi)+\phi, \quad\text{where}\quad \phi=\nu+\beta.
\end{equation}

\emph{Proof of~\eqref{eq:s-rel-t}:}
If $\theta\in[\nu+\beta,\nu)$, then $\theta-\phi\in[0,-\beta)$ and $\theta-\phi+\alpha\in[0,\pi/2)$.
Thus, $S(\theta-\phi)=\theta-\phi+\alpha$ and $T(\theta)=\theta+\alpha=S(\theta-\phi)+\phi$.
Similarly, if $\theta\in[\nu,\nu+\alpha)$, then $\theta-\phi\in[-\beta,\pi/2)$ and $\theta-\phi+\alpha\in[\pi/2,\pi/2+\alpha)$, so that $S(\theta-\phi)=\theta-\phi+\alpha-\pi/2=\theta-\phi+\beta$ and $T(\theta)=\theta+\beta=S(\theta-\phi)+\phi$.

The relation~\eqref{eq:s-rel-t} between $T$ and $S$ shows that all backward orbits of $T$ are dense in $I$.
\end{proof}

\begin{lemma}\label{lem:orbit-interior}
The backward orbit of $T$ from $\nu$, i.e., the set $\{T^{-1}\nu,T^{-2}\nu,T^{-3}\nu,\ldots\}$, is contained in $I\setminus\{\nu\}$.
\end{lemma}

\begin{proof}
By contradiction, if $T^{-k}\nu=\nu$ for some $k\in\Ne_{>0}$, then the backward orbit is periodic (thus not dense), contradicting Lemma~\ref{lem:dense-orbit}.
\end{proof}

Now, we prove our second \emph{key} result, namely that $\Jtilde^\star$ is not differentiable at any $\theta$ in the backward orbit of $\nu$:

\begin{lemma}\label{lem:jtilde-non-differentiable-orbit}
For every $k\in\Ne$, $\Jtilde^\star$ is not differentiable at $T^{-k}\nu$.
\end{lemma}

\begin{proof}
The proof is by induction on $k$.
For $k=0$, the result follows from Lemma~\ref{lem:jtilde-non-differentiable}.
Now, we show it for $k\in\Ne_{>0}$ assuming it holds for $k-1$.
Let $\eta=T^{-k}\nu$.
By Lemma~\ref{lem:orbit-interior}, it holds that $\eta\in[\nu+\beta,\nu)\cup(\nu,\nu+\alpha)$.
We consider two cases.

i) Assume that $\eta\in[\nu+\beta,\nu)$.
By Lemmas~\ref{lem:sys-dp} and~\ref{lem:delta-jtilde-well}, there is a neighborhood $\calV_\eta$ of $\eta$ such that for all $\theta\in\calV_\eta$,
\[
\Jtilde^\star(\theta)=\ctilde(\theta)+\rho^2\Jtilde^\star(\theta+\alpha).
\]
Also, it holds that $T(\eta)=\eta+\alpha$ (since $\eta<\nu$).
Hence, for all $\theta\in\calV_\eta$,
\[
\Jtilde^\star(\theta)=\ctilde(\theta)+\rho^2\Jtilde^\star(\theta-\eta+T^{1-k}\nu).
\]
Since $\ctilde$ is differentiable, this implies that $\Jtilde^\star$ is not differentiable at $\eta$, by using the induction hypothesis (since $\Jtilde^\star$ is not differentiable at $T^{1-k}\nu$).

ii) Assume that $\eta\in(\nu,\nu+\alpha)$.
Then, the same argument as above shows that $\Jtilde^\star$ is not differentiable at $\eta$.

This concludes the proof.
\end{proof}

\begin{corollary}\label{cor:non-differentiable-I}
The function $\Jtilde^\star$ is non-differentiable on the backward orbit of $T$ from $\nu$, which is dense in $I$.
\end{corollary}

\begin{proof}
From Lemmas~\ref{lem:dense-orbit} and~\ref{lem:jtilde-non-differentiable-orbit}.
\end{proof}

Finally, we show that $\Jtilde^\star$ is non-differentiable on a dense subset of $\Re$, and then we deduce that $J^\star$ is non-differentiable on a dense subset of $\Re^2$.

\begin{lemma}\label{lem:non-differentiable-tilde}
The function $\Jtilde^\star$ is non-differentiable on a dense subset of $\Re$.
\end{lemma}

\begin{proof}
We show that $\Jtilde^\star$ is non-differentiable on a dense subset of $(\omega-\pi,\omega)$.
Since for all $\theta\in\Re$ and $k\in\Ze$,
\[
\Jtilde^\star(\theta+k\pi)=J^\star(\pm x(\theta))=J^\star(x(\theta))=\Jtilde^\star(\theta)
\]
(where we used Lemma~\ref{lem:sys-homogeneous} for the second equality), this will conclude the proof.

Let $\calV\subseteq(\omega-\pi,\omega)$ be an open and nonempty interval.
We show that there is $\eta\in\calV$ at which $\Jtilde^\star$ is not differentiable.
We distinguish three cases:

i) Assume that $\calV\cap I\neq\emptyset$.
Then, the result follows from Corollary~\ref{cor:non-differentiable-I}.

ii) Assume that $\calV\subseteq(\omega-\pi,\nu+\beta)$.
Let $\ell\in\Ne$ be the smallest positive integer such that $\calV+\ell\alpha\cap I\neq\emptyset$.
Then, note that, by Lemmas~\ref{lem:sys-dp} and~\ref{lem:delta-jtilde-well}, for all $\theta\in\calV$,
\[
\Jtilde^\star(\theta)=\left\{\sum_{k=0}^{\ell-1} \rho^{2k} \ctilde(\theta+k\alpha)\right\} + \rho^{2\ell}\Jtilde^\star(\theta+\ell\alpha).
\]
By Corollary~\ref{cor:non-differentiable-I} and the assumption on $\ell$, there is $\eta\in\calV$ such that $\Jtilde^\star$ is not differentiable at $\eta+\ell\alpha$.
Using the above, this shows that $\Jtilde^\star$ is not differentiable at $\eta$.

iii) Assume that $\calV\subseteq(\nu+\alpha,\omega)$.
Let $\ell\in\Ne$ be the smallest positive integer such that $\calV+\ell\beta\cap I\neq\emptyset$.
Then, note that, by Lemmas~\ref{lem:sys-dp} and~\ref{lem:delta-jtilde-well}, for all $\theta\in\calV$,
\[
\Jtilde^\star(\theta)=\left\{\sum_{k=0}^{\ell-1} \rho^{2k} \ctilde(\theta+k\beta)\right\} + \rho^{2\ell}\Jtilde^\star(\theta+\ell\beta).
\]
By Corollary~\ref{cor:non-differentiable-I} and the assumption on $\ell$, there is $\eta\in\calV$ such that $\Jtilde^\star$ is not differentiable at $\eta+\ell\beta$.
Using the above, this shows that $\Jtilde^\star$ is not differentiable at $\eta$.

This concludes the proof.
\end{proof}

\begin{corollary}\label{cor:non-differentiable-all}
The function $J^\star$ is non-differentiable on a dense subset of $\Re^2$.
\end{corollary}

\begin{proof}
By Lemma~\ref{lem:non-differentiable-tilde}, we have that $J^\star$ is not differentiable on a dense subset $U$ of $\{x\in\Re^2:\lVert x\rVert=1\}$.
Hence, by Lemma~\ref{lem:sys-homogeneous}, we deduce that $J^\star$ is not differentiable on $\Re U$ which is dense in $\Re^2$.
\end{proof}

This concludes the proof of the non-differentiability of $J^\star$ on a dense subset of $\Re^2$.
Now, let $J^\circ$ be the worst-case value function of $\calA$ with cost $c$.
By leveraging the result for $J^\star$, we show \ifextended easily\footnote{Note that a derivation as for $J^\star$ would also have been possible, but it is easier to build upon the work that has already been done.} \else in the extended version~\cite{berger2025onthedifferentiability} \fi that $J^\circ$ is non-differentiable on a dense subset of $\Re^2$.

\ifextended
\begin{corollary}\label{cor:non-differentiable-all-worst}
The function $J^\circ$ is non-differentiable on a dense subset of $\Re^2$.
\end{corollary}

\begin{proof}
Note that for all $\theta\in\Re$, $\ctilde(\theta)=2-\cos^2(\theta)$.
Hence, for all $\theta\in\Re$, $\ctilde(\theta+\pi/2)=3-\ctilde(\theta)$.
For all $\theta\in\Re$, let $\Jtilde^\circ(\theta)=-\Jtilde^\star(\theta+\pi/2)+3/(1-\rho^2)$.
Then,
\begin{align*}
\Jtilde^\circ(\theta) &= -\Jtilde^\star(\theta+\pi/2)+3/(1-\rho^2) \\
\intertext{(by Lemma~\ref{lem:sys-dp}:)}
&= -\ctilde(\theta+\pi/2)-\rho^2\min\{\Jtilde^\star(\theta+\pi/2+\alpha), \\
&\hspace{2.5cm} \Jtilde^\star(\theta+\pi/2+\beta)\}+3/(1-\rho^2) \\
&= \ctilde(\theta)-3+\rho^2\max\{-\Jtilde^\star(\theta+\pi/2+\alpha),\\
&\hspace{2.5cm}-\Jtilde^\star(\theta+\pi/2+\beta)\}+3/(1-\rho^2) \\
&= \ctilde(\theta)-3+\rho^2\max\{\Jtilde^\circ(\theta+\alpha),\Jtilde^\circ(\theta+\beta)\} \\
&\hspace{2.5cm}+3/(1-\rho^2)-\rho^23/(1-\rho^2) \\
&= \ctilde(\theta)+\rho^2\max\{\Jtilde^\circ(\theta+\alpha),\Jtilde^\circ(\theta+\beta)\}.
\end{align*}
Hence, $J^\circ$ defined by $J^\circ(x)=r^2\Jtilde^\circ(\theta)$ where $x=rx(\theta)$ with $r\geq0$ and $\theta\in\Re$, satisfies~\eqref{eq:dp-worst}.
This implies that $J^\circ$ is the worst-case value function, by classical arguments in dynamic programming (see, e.g.,~\cite[Theorem~3.1]{meyn2022control}).
\end{proof}
\fi

This concludes the proof of Theorem~\ref{thm:non-differentiable} for the case $n=2$.

\subsubsection{Case $n>2$}

\ifextended
Now, we exploit the above to obtain a similar system for $n>2$.
Therefore, consider the matrices
\begin{align*}
\hat{A}_1 &= \rho\left[\begin{array}{ccc}
\cos(\alpha) & -\sin(\alpha) & 0_{1\times(n-2)} \\
\sin(\alpha) & \cos(\alpha) & 0_{1\times(n-2)} \\
0_{(n-2)\times1} & 0_{(n-2)\times1} & 0_{(n-2)\times(n-2)}
\end{array}\right], \\
\hat{A}_2 &= \rho\left[\begin{array}{ccc}
\cos(\beta) & -\sin(\beta) & 0_{1\times(n-2)} \\
\sin(\beta) & \cos(\beta) & 0_{1\times(n-2)} \\
0_{(n-2)\times1} & 0_{(n-2)\times1} & 0_{(n-2)\times(n-2)}
\end{array}\right],
\end{align*}
where $\alpha$, $\beta$ and $\rho$ are as above, and the cost function
\[
c(x)=x_1^2+2x_2^2,
\]
where $x=[x_1,\ldots,x_n]^\top$.
Let $\hat{J}^\star$ and $\hat{J}^\circ$ be the associated optimal and worst-case value functions respectively.

\begin{lemma}\label{lem:link-2-general}
It holds that $\hat{J}^\star(x)=J^\star([x_1,x_2]^\top)$ and $\hat{J}^\circ(x)=J^\circ([x_1,x_2]^\top)$, where $x=[x_1,\ldots,x_n]^\top$ and $J^\star$ and $J^\circ$ are as in the above subsubsection (case $n=2$).
\end{lemma}

\begin{proof}
Straightforward since the dynamics can be separated into two independent dynamics, one for the first two components ($x_1,x_2$) of the state, and one for the last $n-2$ components ($x_3,\ldots,x_n$).
The contribution of the last $n-2$ components to the cost is $0$.
\end{proof}

Lemma~\ref{lem:link-2-general} implies that $J^\star$ and $J^\circ$ are non-differentiable on dense subsets of $\Re^n$.

\begin{corollary}
The functions $\hat{J}^\star$ and $\hat{J}^\circ$ are non-differentiable on dense subsets of $\Re^n$.
\end{corollary}

This concludes the proof of Theorem~\ref{thm:non-differentiable} for the general case.
\else
See extended version~\cite{berger2025onthedifferentiability}.
\fi

\subsubsection*{Rational matrices}

Finally, it remains to show that there are rational matrices of the form~\eqref{eq:A1}--\eqref{eq:A2}.
\ifextended\relax\else For that, we refer the reader to the extended version~\cite{berger2025onthedifferentiability}.\fi

\ifextended
\begin{lemma}\label{lem:rational}
The matrices
\[
A_1 = \rho\left[\begin{array}{cc}
0.8 & -0.6 \\
0.6 & 0.8
\end{array}\right] \;\: \text{and} \;\;
A_2 = \rho\left[\begin{array}{cc}
0.6 & 0.8 \\
-0.8 & 0.6
\end{array}\right]
\]
have the form~\eqref{eq:A1}--\eqref{eq:A2} with $\alpha=\arctan(3/4)\in(\pi/8,3\pi/8)$ and $\beta=\arctan(-4/3)\in(-3\pi/8,-\pi/8)$.
Furthermore, it holds $\beta=\alpha-\pi/2$ and $\alpha/\pi\notin\Qe$.
\end{lemma}

\begin{proof}
The form of $A_1$ and $A_2$ and the fact that $\beta=\alpha-\pi/2$ are straightforward to check.
The fact that $\alpha/\pi\notin\Qe$ follows from Niven's theorem~\cite[Corollary~3.12]{niven1956irrational}.
\end{proof}
\fi

An illustration of the non-differentiability of the optimal value function for the system in~\ifextended Lemma~\ref{lem:rational} \else\cite{berger2025onthedifferentiability} \fi with $\rho=0.9$ is presented in Fig.~\ref{fig:optimal-example}.
Even though the value of $\rho$ (chosen for readability of the figure) is larger that than in the proof of Theorem~\ref{thm:non-differentiable}, we still observe that the number of ``kinks'' (points of non-differentiability) increases with $T$.

\begin{figure}[h]
    \centering
    \begin{tikzpicture}
        \begin{groupplot}[
            group style={
                group size=2 by 2,
                horizontal sep=0.9cm,
                vertical sep=1.0cm,
                xlabels at=edge bottom,
                ylabels at=edge left,
            },
            width=0.55\linewidth,   % <-- key change
            height=0.36\linewidth,
            grid=both,
            xlabel={$\theta$},
            ylabel={$\Jtilde^\star_T(\theta)$},
            xmin=-3.1416,
            xmax=3.1416,
            title style={yshift=-5pt},
        ]

        \nextgroupplot[
            title = {$T=2$}
        ]
        \addplot[
            thick
        ] table[x index=0, y index=1] {cost_1.txt};

        \nextgroupplot[
            title={$T=3$}
        ]
        \addplot[
            thick
        ] table[x index=0, y index=1] {cost_2.txt};

        \nextgroupplot[
            title={$T=4$}
        ]
        \addplot[
            thick
        ] table[x index=0, y index=1] {cost_3.txt};

        \nextgroupplot[
            title={$T=5$}
        ]
        \addplot[
            thick
        ] table[x index=0, y index=1] {cost_4.txt};

        \end{groupplot}
    \end{tikzpicture}
    \caption{Illustration of the non-smoothness of the optimal value function for the switched linear system in~\ifextended Lemma~\ref{lem:rational} \else\cite{berger2025onthedifferentiability} \fi with $\rho=0.9$.
    Each plot shows $\Jtilde^\star_T(\theta)=\min_{\sigma:\Ne\to[m]}\sum_{t=0}^{T-1} c(\xi(t,x(\theta),\sigma))$ the optimal cumulative cost at horizon $T$, starting from $x(\theta)=[\cos(\theta),\sin(\theta)]^\top$.
    The cost function $c$ is as in~\eqref{eq:cost}.
    We see that the number of ``kinks'' (points of non-differentiability) increases with $T$.}
    \label{fig:optimal-example}
\end{figure}

\section{Conclusions}

We showed that the optimal and worst-case value functions of switched linear systems can be highly unsmooth.
Existing results in the literature (e.g.,~\cite{zhang2009onthevalue}) suggested it, but no formal statement or proof were available.
This work addresses this gap, by providing a constructive proof that these functions can be non-differentiable on dense subsets on the state space.
% The implication for their computation is that it would be very hard for generic algorithms to compute them exactly, because those generally search in spaces of piecewise smooth functions.
% Also, the problem of finding the point $x$ for which the optimal or worst-case value functions is minimal appears challenging, because one cannot use derivatives of these functions to find their extrema.
In future work, we plan to identify sufficient conditions under which the value functions are piecewise differentiable.

\ifextended
\addtolength{\textheight}{-10cm}
\else
\addtolength{\textheight}{-12cm}  % This command serves to balance the column lengths
\fi



% \section*{Acknowledgment}

% The preferred spelling of the word ``acknowledgment'' in America is without an ``e'' after the ``g''. Avoid the stilted expression, ``One of us (R. B. G.)~thanks . . .''  Instead, try ``R. B. G. thanks''. Put sponsor acknowledgments in the unnumbered footnote on the first page.

%%%%%%%%%%%%%%%%%%%%%%%%%%%%%%%%%%%%%%%%%%%%%%%%%%%%%%%%%%%%%%%%%%%%%%%%%%%%%%%%

\bibliographystyle{ieeetran}
\bibliography{myrefs}

\end{document}